\documentclass{amsart}
\usepackage[english]{babel}
\usepackage[latin1]{inputenc}

\usepackage{newlfont}
\usepackage{amsmath}
\usepackage{amscd}
\usepackage{amssymb}
\usepackage{amsthm}
\usepackage{mathrsfs}
\usepackage{graphicx}
\usepackage{stmaryrd}
\usepackage{amsmath}
\usepackage{latexsym}
\usepackage{braket}
\usepackage{enumerate}
\usepackage{booktabs}
\usepackage{array}
\usepackage{paralist}
\usepackage{tikz-cd}

\usepackage[a4paper,top=3cm,bottom=3cm,left=3.5cm,right=3.5cm,%
bindingoffset=5mm]{geometry}

\definecolor{red-}{rgb}{1.0,0.0,0.0}

\newtheorem{Theo}{Theorem}[section]
\newtheorem{Lemma}[Theo]{Lemma}
\newtheorem{Cor}[Theo]{Corollary}
\newtheorem{Prop}[Theo]{Proposition}
\newtheorem{Que}[Theo]{Question}
\theoremstyle{definition}
\newtheorem{Def}[Theo]{Definition}
\newtheorem{Rem}[Theo]{Remark}
\newtheorem{Ex}[Theo]{Example}

\title{Generalized Hilbert-Kunz function in graded dimension two}

\author{Holger Brenner and Alessio Caminata}

\address{{\small Holger Brenner, Institut f\"ur Mathematik, Universit\"at Osnabr\"uck, Albrechtstrasse 28a, 49076 Osnabr\"uck, Germany}}
\email{{\small holger.brenner@uni-osnabrueck.de}}

\address{{\small Alessio Caminata, Institut f\"ur Mathematik, Universit\"at Osnabr\"uck, Albrechtstrasse 28a, 49076 Osnabr\"uck, Germany}}
\curraddr{ Institut de Math\'{e}matiques, Universit\'{e} de Neuch\^{a}tel\\Rue Emile-Argand 11, CH-2000 Neuch\^{a}tel, Switzerland}
\email{{\small alessio.caminata@unine.ch}}

\begin{document}


\begin{abstract}
 We prove that the generalized Hilbert-Kunz function of a graded module $M$ over a two-dimensional standard graded normal $K$-domain over an algebraically closed field $K$ of prime characteristic $p$ has the form $gHK(M,q)=e_{gHK}(M)q^{2}+\gamma(q)$, with rational generalized Hilbert-Kunz multiplicity $e_{gHK}(M)$ and a bounded function $\gamma(q)$. Moreover we prove that if $R$ is a $\mathbb{Z}$-algebra, the limit for $p\rightarrow+\infty$ of the generalized Hilbert-Kunz multiplicity $e_{gHK}^{R_p}(M_p)$ over the fibers $R_p$ exists and it is a rational number. 
\end{abstract}

 \maketitle

\section*{Introduction}
Let $R$ be a $d$-dimensional standard graded $K$-domain over a perfect field $K$ of characteristic $p>0$ which is $F$-finite. For every finitely generated $R$-module $M$ and every natural number $e$ we denote by $F^{e*}(M)=M\otimes_R\! ^eR$, the $e$-th iteration of the Frobenius functor given by base change along the Frobenius homomorphism. In particular if $I$ is an ideal of $R$ we have that
$F^{e*}(R/I)\cong R/I^{[p^e]}$. We denote by $q=p^e$ a power of the characteristic. Let $M$ be a graded $R$-module, the function 
\begin{equation*}
 gHK(M,q):=l_R(H^0_{R_{+}}(F^{e*}(M)))
\end{equation*}
and the limit
\begin{equation*}
 e_{gHK}(M):=\lim_{e\rightarrow+\infty}\frac{l_R\left(H^0_{R_+}\left(F^{e*}(M)\right)\right)}{q^d},
\end{equation*}
are called \emph{generalized Hilbert-Kunz function} and \emph{generalized Hilbert-Kunz multiplicity} of $M$ respectively. If $I$ is an $R_+$-primary ideal, then $gHK(R/I,q)$ and $e_{gHK}(R/I)$ coincide with the classical Hilbert-Kunz functions and multiplicity. For a survey on the classical Hilbert-Kunz function and multiplicity see \cite{Hun13}.

\par The generalized Hilbert-Kunz function and multiplicity have been introduced, under a different name and notation, first by Epstein and Yao in \cite{EY11}, and studied in details by Dao and Smirnov in \cite{DS13}, where they prove the existence of $e_{gHK}(M)$ under some assumptions, for example if $M$ is a module over a Cohen-Macaulay isolated singularity.
In the same paper, they study the behaviour of the function $gHK(M,q)$ and compare it with the classical Hilbert-Kunz function. Further study of the generalized Hilbert-Kunz function and multiplicity have been done by Dao and Watanabe in \cite{DW15}, where they compute $e_{gHK}(M)$ if $M$ is a module over a ring of finite Cohen-Macaulay type, or it is an ideal of a normal toric singularity.

\par In this paper we study the function $gHK(M,q)$ for a graded module $M$ over a two-dimensional standard graded normal domain over an algebraically closed field. In \cite{Bre07} Brenner proves that if $I$ is a homogeneous $R_+$-primary ideal, then the Hilbert-Kunz function of $I$ has the following form
\begin{equation*}
 HK(I,q)=e_{HK}(I)q^2+\gamma(q),
\end{equation*}
where $e_{HK}(I)$ is a rational number and $\gamma(q)$ is a bounded function, which is eventually periodic if $K$ is the algebraic closure of a finite field. In \cite[Example 6.2]{DS13}, Dao and Smirnov exhibit numerical evidence that in this setting also the generalized Hilbert Kunz function has the same form. Using an extension of the methods of Brenner, we are able to prove their claim. In fact we obtain in Theorem \ref{TheoremgHKgradedcase} that the generalized Hilbert-Kunz function of a graded module $M$ has the form 
\begin{equation*}
 gHK(M,q)=e_{gHK}(M)q^2+\gamma(q),
\end{equation*}
where $\gamma(q)$ is a bounded function, which is eventually periodic if $K$ is the algebraic closure of a finite field. Moreover we give an explicit formula for $e_{gHK}(M)$ in terms of the Hilbert-Kunz slope of certain locally free sheaves on the projective curve $Y=\mathrm{Proj}R$. As a consequence of this fact we obtain that the generalized Hilbert-Kunz multiplicity $e_{gHK}(M)$ exists and it is a rational number.

\par Furthermore in the last section of the paper, we consider the following problem. Assume that $R$ is a standard graded $\mathbb{Z}$-domain of relative dimension two and $M$ a graded $R$-module. For each prime number $p$, we may consider the reduction $R_p$  of $R$ mod $p$ and the extended module $M_p:=M\otimes_RR_p$. For this module we compute the generalized Hilbert-Kunz multiplicity $e_{gHK}^{R_p}(M_p)$ and we ask whether the limit
\begin{equation*}
 \lim_{p\rightarrow+\infty}e_{gHK}^{R_p}(M_p)
\end{equation*}
exists. Using a result of Trivedi (\cite{Tri07}) we are able to prove (Theorem \ref{limitHKtheorem}) that the previous limit exists, and it is in fact a rational number, assuming that the rings $R_p$ are normal two-dimensional domains for almost all prime numbers.

\par After submitting the first version of this paper, the referee and Asgharzadeh pointed us to a recent paper of Vraciu \cite{Vra16}.
There, she provides another method to prove Theorem \ref{TheoremgHKgradedcase} for ideals by showing that under suitable conditions, which are fulfilled in our situation, the generalized Hilbert-Kunz function of a homogeneous ideal can be expressed as a $\mathbb{Z}$-linear combination of the classical Hilbert-Kunz function of $R_+$-primary ideals. The relevant condition is called $(LC)$ property and was introduced by Hochster and Huneke in \cite{HH90}. 
This condition is known to hold in some special cases, but it is an open problem whether it holds in a more general setting, see \cite{Asg15}.

 \section{Reflexive modules}
 We recall some preliminary facts concerning reflexive modules. 
 Let $R$ be a two-dimensional normal domain with homogeneous maximal ideal $\mathfrak{m}$ and let $U$ be the punctured spectrum of $R$, that is $U=\mathrm{Spec}R\setminus\{\mathfrak{m}\}$.
 \par We denote by $(-)^*$ the functor $\mathrm{Hom}_R(-,R)$. If $M$ is an $R$-module, then the module $M^{**}$ is called the \emph{reflexive hull} of $M$. There is a canonical map 
 \begin{equation*}
  \lambda:M\rightarrow M^{**}.
 \end{equation*}
If $\lambda$ is injective, $M$ is said to be \emph{torsionless}, if $\lambda$ is an isomorphism then $M$ is called \emph{reflexive}. Finitely generated projective modules are reflexive, but the converse does not hold in general. We recall the following geometric characterization of the reflexive hull in the normal situation (cf. \cite[Proposition 3.10]{BD08}):
\begin{equation}\label{geometriccharacterization}
 M^{**}\cong\Gamma(U,\widetilde{M}),
\end{equation}
where $\widetilde{M}$ denotes the coherent sheaf associated to the module $M$. It follows that  the restriction of this sheaf to the punctured spectrum $\widetilde{M}|_U$ coincides with the sheaf $\widetilde{M^{**}}|_U$ on $U$.
Moreover if $M$ is reflexive, the sheaf $\widetilde{M}|_{U}$ is locally free. 

\par The following lemma is a well known fact, see \cite[Proposition 1.4.1]{BH98}, we give a proof here for sake of completeness.
 \begin{Lemma}\label{reflexivecohomologyzero}
 Let $R$ be a normal domain of dimension at least $2$ with homogeneous maximal ideal $\mathfrak{m}$ and let $I$ be a reflexive submodule of $R^n$. Then 
  \begin{equation*}
   H^0_{\mathfrak{m}}(R^n/I)=0.
  \end{equation*}
 \end{Lemma}
\begin{proof}
We consider the short exact sequence $0\rightarrow I\rightarrow R^n\rightarrow R^n/I\rightarrow0$ and we apply the local cohomology functor $H_{\mathfrak{m}}^0(-)$. We obtain a long exact sequence
\begin{equation*}
\cdots\rightarrow H_{\mathfrak{m}}^0(R^n)\rightarrow H_{\mathfrak{m}}^0(R^n/I)\rightarrow H_{\mathfrak{m}}^1(I)\rightarrow \cdots.
\end{equation*}
Since $R^n$ and $I$ are reflexive modules over a normal domain, they have depth at least $2$.
It follows that $H_{\mathfrak{m}}^0(R^n)=H_{\mathfrak{m}}^1(I)=0$, hence $H_{\mathfrak{m}}^0(R^n/I)=0$ too.
\end{proof}

\par We mention also the following result (cf. \cite[Proposition 2.2]{DW15}), concerning the generalized Hilbert-Kunz multiplicity of reflexive ideals.
\begin{Prop}\label{reflexiveproposition}
 Let $R$ be a standard graded domain of dimension $2$, and let $I$ be a homogeneous reflexive ideal of $R$. Then $e_{gHK}(R/I)=0$ if and only if $I$ is principal.
\end{Prop}
The fact that principal reflexive ideals have generalized Hilbert-Kunz multiplicity $0$ holds also in dimension $\geq2$ and is a consequence of Lemma \ref{reflexivecohomologyzero}. In fact, if $I$ is a principal ideal, then $I^{[p^e]}$ is again principal, and in particular reflexive. It follows that $H^0_{R_{+}}(F^{e*}(R/I))\cong H^0_{R_{+}}(R/I^{[p^e]})=0$, so $e_{gHK}(R/I)=0$.

\begin{Lemma}\label{HKclassgroup}
	Let $R$ be a normal $K$-domain of dimension $d\geq2$ over an algebraically closed field $K$ of prime characteristic $p$. 
	Let $I$ be a non-zero homogeneous ideal of $R$ such that $e_{gHK}(R/I)$ exists and let $f\neq0$ be a homogeneous element of $R$. Then
	\begin{equation*}
	e_{gHK}(R/fI)=e_{gHK}(R/I).  
	\end{equation*}
\end{Lemma}

\begin{proof}
	From the short exact sequence $0\rightarrow I\rightarrow R\rightarrow R/I\rightarrow0$ and the corresponding long exact sequence of local cohomology modules with support in $\mathfrak{m}:=R_+$ we obtain that $H_{\mathfrak{m}}^0(R/I)\cong H_{\mathfrak{m}}^1(I)$.
	It follows that the generalized Hilbert-Kunz multiplicity can be seen as
	\begin{equation*}
e_{gHK}(R/I)=\lim_{e\rightarrow+\infty}\frac{l_R\left(H_{\mathfrak{m}}^1(I^{[q]})\right)}{q^d}.
	\end{equation*}
	Then the $R$-module isomorphism $f^qI^{[q]}\cong I^{[q]}$ implies the claim.	
\end{proof}

\begin{Rem}	Let $[I]$ be an element of the divisor class group $\mathrm{Cl}(R)$ of $R$, and let $I$ be a homogeneous reflexive ideal representative of this element. 
If $R$ is a standard graded normal $K$-domain of dimension $2$, with $K$ algebraically closed and of positive characteristic, we obtain a function $e_{gHK}(-):\mathrm{Cl}(R)\rightarrow\mathbb{Q}$, $[I]\mapsto e_{gHK}(R/I)$. 
Thanks to Theorem \ref{TheoremgHKgradedcase} and Lemma \ref{HKclassgroup}, this function is well-defined and we have $e_{gHK}([R])=0$. 
	This does not mean that the generalized Hilbert-Kunz multiplicity for all ideals $I$ which are invertible on the punctured spectrum depends only on $[I]$. For example, the homogeneous maximal ideal $\mathfrak{m}$ and its reflexive hull $\mathfrak{m}^{**}=R$ define the same element in the class group, but $e_{gHK}(R/\mathfrak{m})=e_{HK}(\mathfrak{m})\neq0$ in general, while $e_{gHK}(R/R)=0$. 
Moreover Proposition \ref{reflexiveproposition} implies that the preimage of $0$ is trivial. This does not mean that the function $e_{gHK}(-)$ is injective, since in general it is not a group homomorphism as in the case of Example \ref{exampleidealofapoint}.
\end{Rem}
\par Therefore the following question makes sense.
\begin{Que}
Given two homogeneous reflexive ideals $I$ and $J$, is there a formula for $e_{gHK}([IJ])$ in terms of $e_{gHK}([I])$ and $e_{gHK}([J])$?
\end{Que}

\section{The Hilbert-Kunz slope}
\par Let $Y$ be a smooth projective curve over an algebraically closed field with a very ample invertible sheaf of degree $\deg\mathcal{O}_Y(1)=\deg Y$. We recall some classical notions of vector bundles and some definitions from \cite{Bre06} and \cite{Bre07}. We refer to those papers for further details and explanations.
\par Let $\mathcal{S}$ be a locally free sheaf of rank $r$ over $X$. The degree of $\mathcal{S}$ is defined as the degree of the corresponding determinant line bundle $\mathrm{deg}\mathcal{S}=\deg\bigwedge^r\mathcal{S}$. The slope of $\mathcal{S}$ is $\mu(\mathcal{S})=\deg\mathcal{S}/r$. The degree is additive on short exact sequences and moreover $\mu(\mathcal{S}\otimes\mathcal{T})=\mu(\mathcal{S})+\mu(\mathcal{T})$.
\par The sheaf $\mathcal{S}$ is called \emph{semistable} if for every locally free subsheaf $\mathcal{T}\subseteq\mathcal{S}$ the inequality $\mu(\mathcal{T})\leq\mu(\mathcal{S})$ holds. If the strict inequality $\mu(\mathcal{T})<\mu(\mathcal{S})$ holds for every proper subsheaf $\mathcal{T}\subset\mathcal{S}$, then $\mathcal{S}$ is called \emph{stable}.
\par For any locally free sheaf $\mathcal{S}$ on $Y$ there exists a unique filtration, called \emph{Harder-Narasimham filtration}, $\mathcal{S}_1\subseteq\dots\subseteq\mathcal{S}_t=\mathcal{S}$ with the following properties:
\begin{compactitem}
 \item $\mathcal{S}_k$ is locally free,
 \item $\mathcal{S}_k/\mathcal{S}_{k-1}$ is semistable,
 \item $\mu(\mathcal{S}_k/\mathcal{S}_{k-1})>\mu(\mathcal{S}_{k+1}/\mathcal{S}_{k})$.
\end{compactitem}
\par If the base field has positive characteristic, we can consider the absolute Frobenius morphism $F:Y\rightarrow Y$ on the curve  and its iterates $F^e$. In general the pull-back via $F^e$ of the Harder-Narasimham filtration of $\mathcal{S}$ is not the Harder-Narasimham filtration of $F^{e*}\mathcal{S}$, since the quotients $F^{e*}(\mathcal{S}_k)/F^{e*}(\mathcal{S}_{k-1})$ need not to be semistable.
\par In \cite{Lan04}, Langer proved that for $q\gg0$ there exists the so-called  \emph{strong Harder-Narasimham} filtration of $F^{e*}(\mathcal{S})$. In fact there exists a natural number $e_0$ such that the Harder-Narasimham filtration of $F^{e_0*}(\mathcal{S})$
\begin{equation*}
 0\subseteq \mathcal{S}_{e_0,1}\subseteq\dots\subseteq\mathcal{S}_{e_0,t}=F^{e_0*}(\mathcal{S})
\end{equation*}
has the property that the quotients $F^{e*}(\mathcal{S}_{e_0,k})/F^{e*}(\mathcal{S}_{e_0,k-1})$ of the pullback along $F^{e}$ are semistable.
Thus for $e\geq e_0$ we have $F^{e*}(\mathcal{S})=F^{(e-e_0)*}(F^{e_0*}(\mathcal{S}))$, and the Harder-Narasimham filtration of $F^{e*}(\mathcal{S})$ is given by
\begin{equation*}
 F^{(e-e_0)*}(\mathcal{S}_{e_0,1}) \subseteq\dots\subseteq F^{(e-e_0)*}(\mathcal{S}_{e_0,t})=F^{e*}(\mathcal{S}).
\end{equation*}
For ease of notation we put $\mathcal{S}_{e,k}:=F^{(e-e_0)*}(\mathcal{S}_{e_0,k})$ for every $e\geq e_0$ and $0\leq k\leq t$.
The length $t$ of such a sequence and the ranks of the quotients $\mathcal{S}_{e,k}/\mathcal{S}_{e,k-1}$ are independent of $e$, while the degrees are not. We define the following rational numbers:
\begin{itemize}
 \item $\bar{\mu}_k=\bar{\mu}_k(\mathcal{S})=\displaystyle\frac{\mu(\mathcal{S}_{e,k}/\mathcal{S}_{e,k-1})}{p^e}$, where $\mu(-)$ denotes the usual slope of the bundle,
 \item $r_k=\mathrm{rank}(\mathcal{S}_{e,k}/\mathcal{S}_{e,k-1})$,
 \item $\nu_k=-\displaystyle\frac{\bar{\mu}_k}{\mathrm{deg}Y}$.
\end{itemize}

\begin{Rem}
We point out that the numbers $\bar{\mu}$, $r_k$ and $\nu_k$ are rational and independent from $e$  for $e\gg0$. In fact we have that $\sum_{k=1}^tr_k\mu(\mathcal{S}_{e,k}/\mathcal{S}_{e,k-1})=\deg(F^{e*}\mathcal{S})=p^e\deg\mathcal{S}$, which implies the relation
\begin{equation*}
\sum_{k=1}^tr_k\bar{\mu}_k=\deg\mathcal{S}. 
\end{equation*}
\end{Rem}

\begin{Def}
 Let $\mathcal{S}$ be a locally free sheaf over a projective curve over an algebraically closed field of prime characteristic and let $\bar{\mu}$ and $r_k$ be as above. The \emph{Hilbert-Kunz slope} of $\mathcal{S}$ is the rational number
 \begin{equation*}
  \mu_{HK}(\mathcal{S})=\sum_{k=1}^tr_k\bar{\mu}_k^2.
 \end{equation*}
\end{Def}
This notion was introduced by the first author in \cite{Bre07}, where he also proves Theorem \ref{theoremalternatingsum} below.

\begin{Ex}\label{slopelinebundle}
 Let $\mathcal{L}$ be a line bundle, then $\mathcal{L}$ is semistable of slope $\mu(\mathcal{L})=\deg{\mathcal{L}}$. The pullback along Frobenius is again a line bundle: $F^{e*}\mathcal{L}=\mathcal{L}^{q}=\mathcal{L}^{\otimes q}$, with $q=p^e$. It follows that $0\subseteq\mathcal{L}$ is the strong Harder-Narasimham filtration of $\mathcal{L}$ and the Hilbert-Kunz slope is just:
 \begin{equation*}
  \mu_{HK}(\mathcal{L})=(\deg\mathcal{L})^2.
 \end{equation*}
\end{Ex}

\begin{Ex}\label{slopesum}
 Let $d_1<d_2<\cdots<d_m$ be non-negative integers and let $\mathcal{T}:=\bigoplus_{i=1}^m\mathcal{O}(-d_i)^{\oplus r_i}$, where $\mathcal{O}:=\mathcal{O}_Y$ and $r_i\in\mathbb{N}$. The Harder-Narasimham filtration of $\mathcal{T}$ is
 \begin{equation*}
  0\subseteq\mathcal{O}(-d_1)^{\oplus r_1}\subseteq\mathcal{O}(-d_1)^{\oplus r_1}\oplus\mathcal{O}(-d_2)^{\oplus r_2}\subseteq\dots\subseteq\bigoplus_{i=1}^m\mathcal{O}(-d_i)^{\oplus r_i}.
 \end{equation*}
The quotients are direct sums of line bundles of the same degree, so their pullbacks under Frobenius are semistable. Hence this is also the strong Harder-Narasimham filtration of $\mathcal{T}$ with invariants $r_k$, and $\bar{\mu}_k=\deg\mathcal{O}(-d_k)=-d_k\deg\mathcal{O}_Y(1)=-d_k\deg Y$. Then  the Hilbert-Kunz slope of $\mathcal{T}$ is 
\begin{equation*}
 \mu_{HK}(\mathcal{T})=(\deg Y)^2\sum_{k=1}^mr_kd_k^2.
\end{equation*}
\end{Ex}

\begin{Theo}[Brenner \cite{Bre07}]\label{theoremalternatingsum}
 Let $Y$ denote a smooth projective curve of genus $g$ over an algebraically closed field of positive characteristic $p$ and let $q=p^e$ for a non-negative integer $e$. Let $0\rightarrow\mathcal{S}\rightarrow\mathcal{T}\rightarrow\mathcal{Q}\rightarrow0$ denote a short exact sequence of locally free sheaves on $Y$. Then the following hold.
 \begin{enumerate}
  \item For every non-negative integer $e$ the alternating sum of the dimensions of the global sections is
      \begin{equation*}
  \begin{split}
   \sum_{m\in\mathbb{Z}}\left(h^0(F^{e*}\mathcal{S}(m))-h^0(F^{e*}\mathcal{T}(m))+h^0(F^{e*}\mathcal{Q}(m))\right)\\
   =\frac{q^2}{2\deg Y}\left(\mu_{HK}(\mathcal{S})-\mu_{HK}(\mathcal{T})+\mu_{HK}(\mathcal{Q})\right)+O(q^0).
  \end{split}
  \end{equation*}

\item If the field is the algebraic closure of a finite field, then the $O(q^0)$-term is eventually periodic.
 \end{enumerate}
\end{Theo}

The alternating sum in Theorem \ref{theoremalternatingsum} is in fact a finite sum for every $q$. For $m\ll 0$ the locally free sheaves have no global sections, so all the terms are $0$ and for $m\gg0$ we have $H^1(Y,F^{e*}\mathcal{S}(m))=0$ and the sum is $0$. Moreover the sum is the dimension of the cokernel
\begin{equation*}
 \sum_{m\in\mathbb{Z}}\dim(\Gamma(Y,F^{e*}\mathcal{Q}(m)))/\mathrm{im}(\Gamma(Y,F^{e*}\mathcal{T}(m))).
\end{equation*}

\par In \cite{Bre07}, Brenner uses Theorem \ref{theoremalternatingsum} to prove that the Hilbert-Kunz function of a homogeneous $R_+$-primary ideal $I$ in a normal two-dimensional standard-graded $K$-domain $R$ has the following form
\begin{equation*}
 HK(I,q)=e_{HK}(I)q^2+\gamma(q),
\end{equation*}
where $e_{HK}(I)$ is a rational number and $\gamma(q)$ is a bounded function, which is eventually periodic if $K$ is the algebraic closure of a finite field. In particular if $I$ is generated by homogeneous elements $f_1,\dots,f_n$ of degrees $d_1,\dots,d_n$ and $r_k$, $\bar{\mu}_k$ denote the numerical invariants of the strong Harder-Narasimham filtration of the locally free sheaf $\mathrm{Syz}(f_1,\dots,f_n)$ on the curve $Y=\mathrm{Proj}R$ then the Hilbert-Kunz multiplicity of $I$ is given by
\begin{equation*}
 e_{HK}(I)=\frac{1}{2\deg Y}\left(\sum_{k=1}^tr_k\bar{\mu}_k^2-(\deg Y)^2\sum_{i=1}^nd_i^2\right).
\end{equation*}

\par In Section 3 we apply this method to deduce a similar result for the generalized Hilbert-Kunz function and answer a question of Dao and Smirnov \cite[Example 6.2]{DS13}.

\section{The generalized Hilbert-Kunz function in dimension $2$}

\begin{Lemma}\label{lenghtquotientlemma2}
 Let $R$ be a two-dimensional normal $K$-domain of positive characteristic $p$ with homogeneous maximal ideal $\mathfrak{m}$. We denote by  $U=\mathrm{Spec}R\setminus\{\mathfrak{m}\}$ the punctured spectrum.
 Let $M$ be a finitely generated graded $R$-module with a presentation 
 \begin{equation}\label{presentationofM}
 0\rightarrow I\rightarrow R^n\rightarrow M\rightarrow0.
  \end{equation}
 Let $J=I^{**}$ be the reflexive hull of $I$ (considered inside $R^n$) and $\mathcal{L}$ the coherent sheaf corresponding to $J$ on $U$, that is $\mathcal{L}=\widetilde{J}|_{U}$, then
 \begin{equation}\label{reductionequation}
 \begin{split}
  gHK(M,q)&=l_R\left(\Gamma(U,F^{e*}\mathcal{L})/\mathrm{im}F^{e*}I\right)\\
  &=l_R\left((F^{e*}J)^{**}/\mathrm{im}F^{e*}I\right),
 \end{split}
 \end{equation}
where $q=p^e$ and $\mathrm{im}F^{e*}I$ denotes the image of the map $F^{e*}I\rightarrow F^{e*}R^n\cong R^n$.
\end{Lemma}

Before proving the lemma, we explain the right hand side of the equality \eqref{reductionequation}. 
\par First of all, in virtue of \eqref{geometriccharacterization} we have $\Gamma(U,F^{e*}\mathcal{L})=(F^{e*}J)^{**}$, so the second equality is clear. Then the inclusion $I\hookrightarrow R^n$ factors through the reflexive module $J$. Applying the Frobenius functor to these maps we get a commutative diagram
\begin{equation}\label{frobeniuscommutative}
   \begin{tikzcd}
&F^{e*}I \arrow{r}\arrow{d}
&F^{e*}R^n\cong R^n \\
&F^{e*}J\arrow{ru}.
\end{tikzcd}
\end{equation}
Since the functor $F^{e*}$ is not left exact in general, the maps in \eqref{frobeniuscommutative} are not injective, for this reason we consider the image $\mathrm{im}F^{e*}I\subseteq{R}^n$.
\par Since $R$ is normal, $U$ is smooth and the absolute Frobenius morphism $F^{e}:U\rightarrow U$ is exact on $U$. So we pullback along $F^{e}$ the inclusion $\mathcal{L}\hookrightarrow\mathcal{O}^n_U$ and we take sections on $U$ obtaining the inclusion
\begin{equation*}
 \Gamma(U,F^{e*}\mathcal{L})\hookrightarrow \Gamma(U,F^{e*}\mathcal{O}^n_U)\cong \Gamma(U,\mathcal{O}^n_U)=R^n.
\end{equation*}
Therefore the quotient $\Gamma(U,F^{e*}\mathcal{L})/\mathrm{im}F^{e*}I$ is a quotient of submodules of $R^n$.

\begin{proof}
 We apply the functor $F^{e*}$ to the short exact sequence \eqref{presentationofM} and we get $F^{e*}I\rightarrow R^n\rightarrow F^{e*}M\rightarrow0$. Therefore we have
 \begin{equation}\label{frobeniusequality1}
  F^{e*}M=R^n/\mathrm{im}F^{e*}I.
 \end{equation}
Then we consider the short exact sequence
\begin{equation*}
 0\rightarrow \Gamma(U,F^{e*}\mathcal{L})/\mathrm{im}F^{e*}I\rightarrow R^n/\mathrm{im}F^{e*}I \rightarrow R^n/\Gamma(U,F^{e*}\mathcal{L})\rightarrow0.
\end{equation*}
Taking local cohomology yields
\begin{equation*}
 0\rightarrow H_{\mathfrak{m}}^0\left(\Gamma(U,F^{e*}\mathcal{L})/\mathrm{im}F^{e*}I\right)\rightarrow H_{\mathfrak{m}}^0\left(R^n/\mathrm{im}F^{e*}I\right)\rightarrow H_{\mathfrak{m}}^0\left(R^n/\Gamma(U,F^{e*}\mathcal{L})\right).
\end{equation*}
The module $\Gamma(U,F^{e*}\mathcal{L})$ is reflexive by \eqref{geometriccharacterization}, then by Lemma \ref{reflexivecohomologyzero} the last module of the previous sequence is $0$. So we get the following isomorphism
\begin{equation}\label{frobeniusequality2}
\begin{split}
H_{\mathfrak{m}}^0\left(R^n/\mathrm{im}F^{e*}I\right)&\cong  H_{\mathfrak{m}}^0\left(\Gamma(U,F^{e*}\mathcal{L})/\mathrm{im}F^{e*}I\right)\\
&= \Gamma(U,F^{e*}\mathcal{L})/\mathrm{im}F^{e*}I.
\end{split}
\end{equation}
The last equality holds because the module $\Gamma(U,F^{e*}\mathcal{L})/\mathrm{im}F^{e*}I$ has support in $\mathfrak{m}$, since the sheaves  $\mathcal{L}$ and $\widetilde{I}$ coincide on $U$. Then the desired formula follows from \eqref{frobeniusequality1} and \eqref{frobeniusequality2}.
\end{proof}

\begin{Theo}\label{TheoremgHKgradedcase}
 Let $R$ be a two-dimensional normal standard graded $K$-domain over an algebraically closed field $K$ of prime characteristic $p$ and let $M$ be a finitely generated graded $R$-module.
  Then the generalized Hilbert-Kunz function of $M$ has the form
 \begin{equation*}
  gHK(M,q)=e_{gHK}(M)q^2+\gamma(q),
 \end{equation*}
where  $e_{gHK}(M)$ is a rational number and $\gamma(q)$ is a bounded function.
\par Moreover if $K$ is the algebraic closure of a finite field, then $\gamma(q)$ is an eventually periodic function. In particular, given a graded presentation of $M$
 \begin{equation*}
\bigoplus_{i=1}^nR(-d_i)\xrightarrow{\psi}\bigoplus_{j=1}^mR(-e_j)\rightarrow M\rightarrow0  
 \end{equation*}
 and the corresponding short exact sequence of locally free sheaves on the curve $Y=\mathrm{Proj}R$
 \begin{equation}\label{syzygysequence2}
  0\rightarrow\mathcal{S}:=\widetilde{\mathrm{ker}{\psi}}\rightarrow\mathcal{T}:=\bigoplus_{i=1}^n\mathcal{O}_Y(-d_i)\rightarrow\mathcal{Q}:=\widetilde{\mathrm{im}\psi}\rightarrow0,
 \end{equation}
then the generalized Hilbert-Kunz multiplicity of $M$ is
\begin{equation*}
 e_{gHK}(M)=\frac{1}{2\deg Y}\left(\mu_{HK}(\mathcal{S})-(\deg Y)^2\sum_{i=1}^nd_i^2+\mu_{HK}(\mathcal{Q})\right).
\end{equation*}
\end{Theo}

\begin{proof}
Let $u_1,\dots,u_m$ be homogeneous generators of $M$ of degrees $e_1,\dots,e_m$ respectively and let
\begin{equation*}
 0\rightarrow I\rightarrow\bigoplus_{j=1}^mR(-e_j)\xrightarrow{u_1,\dots,u_m} M\rightarrow0
\end{equation*}
be the corresponding short exact sequence.  Let $f_1,\dots, f_n$ be homogeneous generators of $I$ of degrees $d_1,\dots,d_n$ respectively and let
\begin{equation*}
 0\rightarrow N\rightarrow\bigoplus_{i=1}^{n}R(-d_i)\xrightarrow{f_1,\dots,f_n}I\rightarrow0
\end{equation*}
be the corresponding graded short exact sequence. This last sequence induces the short exact sequence  \eqref{syzygysequence2} on $Y$, and the short exact sequence
\begin{equation}\label{syzygysequence3}
 0\rightarrow \mathcal{E}:=\widetilde{N}|_{U}\rightarrow\mathcal{F}:=\bigoplus_{i=1}^n\mathcal{O}_U(-d_i)\rightarrow\widetilde{I}|_U\rightarrow0
\end{equation}
on the punctured spectrum $U$.
The modules $N$ and $I$ are submodules of finite free $R$-modules, so they are torsion-free. It follows that the corresponding sheaves $\mathcal{E}$ and $\widetilde{I}|_U$ on $U$ are locally free, since $U$ is regular.
Moreover if $J=I^{**}$ is the reflexive hull of $I$ and $\mathcal{L}$ the coherent sheaf corresponding to $J$ on $U$, we have that $\mathcal{L}=\widetilde{I}|_U$ as sheaves on $U$. 
\par From Lemma \ref{lenghtquotientlemma2}, the generalized Hilbert-Kunz function of $M$ is given by
 \begin{equation*}
  gHK(M,q)=l_R\left(\Gamma(U,F^{e*}\mathcal{L})/\mathrm{im}F^{e}I\right)=\sum_{m\in\mathbb{Z}}l_R\left(\left(\Gamma(U,F^{e*}\mathcal{L})/\mathrm{im}F^{e}I\right)_m\right).
 \end{equation*}

 \par To compute the last sum, we  consider the sequence \eqref{syzygysequence2}, we pull it back along the $e$-th absolute Frobenius morphism on $Y$ and we tensor with $\mathcal{O}_Y(m)$, for an integer $m$. We obtain an exact sequence
\begin{equation*}
 0\rightarrow F^{e*}\mathcal{S}(m)\rightarrow F^{e*}\mathcal{T}(m)\rightarrow F^{e*}\mathcal{Q}(m)\rightarrow0.
\end{equation*}
Then we take global sections $\Gamma(Y,-)$ of the last sequence and we get
\begin{equation*}
 0\rightarrow \Gamma(Y,F^{e*}\mathcal{S}(m))\rightarrow\Gamma(Y,F^{e*}\mathcal{T}(m))\xrightarrow{\varphi_m}\Gamma(Y,F^{e*}\mathcal{Q}(m))\rightarrow\dots
\end{equation*}
We are interested in the cokernel of the map $\varphi_m$. Its image is clearly $(\mathrm{im}F^{e}I)_m$. For the evaluation of the sheaf $F^{e*}\mathcal{Q}(m)$ on $Y$, we consider the sequences \eqref{syzygysequence2} and \eqref{syzygysequence3}, we obtain
\begin{equation*}
 \Gamma(Y,F^{e*}\mathcal{Q}(m))\cong\Gamma(U,F^{e*}\mathcal{L})_m.
\end{equation*}
So we get  $\mathrm{Coker}(\varphi_m)=\Gamma(U,F^{e*}\mathcal{L})_m/(\mathrm{im}F^{e}I)_m=\left(\Gamma(U,F^{e*}\mathcal{L})/\mathrm{im}F^{e}I\right)_m$. It follows that 
\begin{equation*}
gHK(M,q)=\sum_{m\in\mathbb{Z}}\mathrm{Coker}(\varphi_m).
\end{equation*}
We compute the last sum with Theorem \ref{theoremalternatingsum} and we obtain the desired formula for the generalized Hilbert-Kunz function.
\par For the generalized Hilbert-Kunz multiplicity, it is enough to notice that $\mu_{HK}(\mathcal{T})=(\deg Y)^2\sum_{i=1}^nd_i^2$ by Example \ref{slopesum}.
\end{proof}

\begin{Cor}\label{corollaryHKideal}
 Let $I$ be a non-zero ideal generated by homogeneous elements $f_1,\dots,f_n$ of degrees $d_1,\dots,d_n$ respectively, and let $d$ be the degree of the ideal sheaf associated to $I$ on $Y=\mathrm{Proj} R$.
 Then the generalized Hilbert-Kunz multiplicity of $R/I$ is given by
  \begin{equation*}
 e_{gHK}(R/I)=\frac{1}{2\deg Y}\left(\sum_{k=1}^tr_k\bar{\mu}_k^2-(\deg Y)^2\sum_{i=1}^nd_i^2+d^2\right),
\end{equation*}
where $r_k$,  $\bar{\mu}_k$ and $t$ are the numerical invariants of the strong Harder-Narasimham filtration of the syzygy bundle $\mathrm{Syz}(f_1,\dots,f_n)$.
\end{Cor}

\begin{proof}
In this case the presenting sequence of $R/I$ is just $0\rightarrow I\rightarrow R\rightarrow R/I\rightarrow 0$ and the sequence \eqref{syzygysequence2} is then
\begin{equation*}
0\rightarrow\mathrm{Syz}(f_1,\dots,f_n)\rightarrow\bigoplus_{i=1}^n\mathcal{O}_Y(-d_i)\xrightarrow{f_1,\dots,f_n}\mathcal{Q}\rightarrow0.
 \end{equation*}
So by Theorem \ref{TheoremgHKgradedcase}, the generalized Hilbert-Kunz multiplicity of $R/I$ is given by 
\begin{equation*}
 \frac{1}{2\deg Y}\left(\mu_{HK}(\mathrm{Syz}(f_1,\dots,f_n))-(\deg Y)^2\sum_{i=1}^nd_i^2+\mu_{HK}(\mathcal{Q})\right).
\end{equation*}
In this situation $\mathcal{Q}$ is a line bundle, so by Example \ref{slopelinebundle}, $\mu_{HK}(\mathcal{Q})=(\deg \mathcal{Q})^2=d^2$, and by definition $\mu_{HK}(\mathrm{Syz}(f_1,\dots,f_n))=\sum_{k=1}^tr_k\bar{\mu}_k^2$.
\end{proof}

\begin{Ex}\label{exampleprincipalideal}
 Let $h$ be a homogeneous element of degree $a>0$, and let $I=(h)$. Then we have $0\rightarrow \mathcal{O}_Y(-a)\xrightarrow{\simeq}\mathcal{L}\rightarrow0$, hence $\mathrm{Syz}(h)=0$. Since $I$ is principal, the degree of the ideal sheaf associated to $I$ is $a\cdot\deg Y$. The generalized Hilbert-Kunz multiplicity is then
 \begin{equation*}
  e_{gHK}(R/I)=\frac{1}{2\deg Y}\left(-(\deg Y)^2a^2+(\deg Y)^2a^2\right)=0,
 \end{equation*}
 in accordance with  Proposition \ref{reflexiveproposition}.
\end{Ex}

\begin{Ex}\label{exampleideal2elements}
 Let $I$ be a prime ideal of height one generated by two homogeneous elements $f$ and $g$ of degrees $a$ and $b$ respectively. Then the syzygy sequence is
 \begin{equation*}
  0\rightarrow \mathrm{Syz}(f,g)\rightarrow \mathcal{O}_Y(-a)\oplus\mathcal{O}_Y(-b)\rightarrow\mathcal{Q}\rightarrow0,
 \end{equation*}
with $\mathcal{Q}$ the line bundle of say degree $d$ associated to the ideal $I$. From this sequence we see that the syzygy bundle has rank one and degree $\deg Y(-a-b)-d$. Therefore we have
\begin{equation*}
\begin{split}
e_{gHK}(R/I)&=\frac{1}{2\deg Y}\left(\left(\deg Y(-a-b)-d\right)^2-(\deg Y)^2(a^2+b^2)+d^2\right)\\ 
&= \frac{1}{2\deg Y}\Big(2d^2+(\deg Y)^2(a^2+b^2)+2ab(\deg Y)^2+2d\deg Y(a+b)\\
&-(\deg Y)^2(a^2+b^2)\Big)\\
&=\frac{1}{\deg Y}\left(d^2+ab(\deg Y)^2+d(a+b)\deg Y\right)\\
&=\frac{d^2}{\deg Y}+ab\deg Y+d(a+b).
\end{split}
\end{equation*}
\end{Ex}

\begin{Ex}\label{exampleidealofapoint}
Let $P$ be a point of the smooth projective curve $Y=\mathrm{Proj}R$, and let $I\subseteq R$ be the corresponding homogeneous prime ideal of height one.
The ideal $I$ is minimally generated by two linear forms $f$ and $g$.
In fact, $f$ and $g$ correspond to two hyperplanes in the projective space where $Y$ is embedded which meet transversally in $P$.
Then, using the notations of Example \ref{exampleideal2elements} we have $a=b=1$, and $d=-1$, since the line bundle $\mathcal{Q}$ associated to the ideal $I$ is a subsheaf of $\mathcal{O}_Y$. So we obtain
\begin{equation*}
e_{gHK}(R/I)=\frac{1}{\deg Y}+\deg Y-2=\frac{(\deg Y-1)^2}{\deg Y}.
\end{equation*}
\end{Ex}

\section{The limit of generalized Hilbert-Kunz multiplicity}
Let $R$ be a standard graded domain flat over $\mathbb{Z}$ such that almost all fiber rings $R_p=R\otimes_{\mathbb{Z}}\mathbb{Z}/p\mathbb{Z}$ are geometrically normal two-dimensional domains. 
We define $R_0:=R\otimes_{\mathbb{Z}}\mathbb{Q}$ and the corresponding projective curve $Y_0:=\mathrm{Proj}R_0$ over the generic point. We denote by $Y_p:=\mathrm{Proj}R_p$, the projective curve over the prime number $p$. This is a smooth projective curve for almost all primes. If $\mathcal{S}$ is a sheaf over the curve $Y:=\mathrm{Proj}R$ we will denote by $\mathcal{S}_p$ (respectively $\mathcal{S}_0$) the corresponding restriction to the curves $Y_p$ (resp. $Y_0$).

\begin{Rem}
 In our setting the curves $Y_0$ and $Y_p$ are not defined over an algebraically closed field. However, we may consider the curves $\overline{Y}_0:=Y_0\times_{\mathbb{Q}}\overline{\mathbb{Q}}$ and $\overline{Y}_p:=Y_p\times_{\mathbb{Z}/p\mathbb{Z}}\overline{\mathbb{Z}/p\mathbb{Z}}$, which are smooth projective curves over the algebraic closures.
 We can consider the definitions of degree, slope, semistable, HN filtration and strong HN filtration for those curves and transfer them to the original curves $Y_0$ and $Y_p$. Therefore we will move to the algebraic closure and back whenever this is convenient.
 \end{Rem}

\par Let $M$ be a graded $R$-module. For every prime $p$ we can consider the reduction to characteristic $p$, $M_p:=M\otimes_RR_p\cong M\otimes_{\mathbb{Z}}\mathbb{Z}/p\mathbb{Z}$, and compute the generalized Hilbert-Kunz multiplicity $e_{gHK}^{R_p}(M_p)$ of the $R_p$-module $M_p$.
Since the projective curve $Y_p$ is smooth for almost all primes $p$,  by Theorem \ref{TheoremgHKgradedcase} we know that $e_{gHK}^{R_p}(M_p)$ exists and that it is rational for these primes. We are interested in the behaviour of $e_{gHK}^{R_p}(M_p)$ for $p\rightarrow+\infty$.

\par We introduce the following characteristic zero version of the Hilbert-Kunz slope.

\begin{Def}
 Let $\mathcal{S}$ be a locally free sheaf over a projective curve over an algebraically closed field of characteristic zero and let $\mathcal{S}_1\subseteq\dots\subseteq\mathcal{S}_t=\mathcal{S}$  be the Harder-Narasimham filtration of $\mathcal{S}$. For every $k=1,\dots,t$ we set $\bar{\mu}_k=\bar{\mu}_k(\mathcal{S})=\mu(\mathcal{S}_{k}/\mathcal{S}_{k-1})$ and $r_k=\mathrm{rank}(\mathcal{S}_{k}/\mathcal{S}_{k-1})$.
 The \emph{Hilbert-Kunz slope} of $\mathcal{S}$ is the rational number
 \begin{equation*}
  \mu_{HK}(\mathcal{S})=\sum_{k=1}^tr_k\bar{\mu}_k^2.
 \end{equation*}
\end{Def}

The name Hilbert-Kunz slope is justified by the following result of Trivedi (cf. \cite[Lemma 1.14]{Tri07}). 

\begin{Lemma}[Trivedi]\label{trivedilemma}
Let $h\in\mathbb{Z}_+$, let $Y$ be a smooth projective curve over $\mathrm{Spec}\mathbb{Z}_h$, and let $\mathcal{S}$ be a locally free sheaf over $Y$. We denote by $\mathcal{S}_0$ and $\mathcal{S}_p$ the restrictions of $\mathcal{S}$ to $Y_0$ and to $Y_p$, for $p\nmid h$. Then
\begin{equation*}
\lim_{\substack{p\rightarrow+\infty\\ p\nmid h}}\mu_{HK}(\mathcal{S}_p)=\mu_{HK}(\mathcal{S}_0). 
\end{equation*}
\end{Lemma}

\begin{Theo}\label{limitHKtheorem}
Let $R$ be a standard graded domain flat over $\mathbb{Z}$ such that almost all fiber rings $R_p=R\otimes_{\mathbb{Z}}\mathbb{Z}/p\mathbb{Z}$ are geometrically normal two-dimensional domains. Let $M$ be a graded $R$-module with a graded presentation
\begin{equation*}
 \bigoplus_{i=1}^nR(-d_i)\xrightarrow{\psi}\bigoplus_{j=1}^mR(-e_j)\rightarrow M\rightarrow0,  
\end{equation*}
and corresponding short exact sequence of locally free sheaves $0\rightarrow\mathcal{S}_0\rightarrow\mathcal{T}_0\rightarrow\mathcal{Q}_0\rightarrow0$
on the generic fiber $Y_0=\mathrm{Proj}R_0$, with notations as above. Then the limit 
\begin{equation*}
 \lim_{p\rightarrow+\infty}e_{gHK}^{R_p}(M_p)
\end{equation*}
exists and it is equal to the rational number
\begin{equation*}
  \frac{1}{2\deg Y_0}\left(\mu_{HK}(\mathcal{S}_0)-\mu_{HK}(\mathcal{T}_0)+\mu_{HK}(\mathcal{Q}_0)\right).
\end{equation*}
\end{Theo}

\begin{proof}
 Let $u_1,\dots,u_m$ be homogeneous generators of $M$ as $R$-module, and let $f_1,\dots,f_n$ be homogeneous generators of $I:=\mathrm{Syz}(u_1,\dots,u_m)$. We obtain two short exact sequences
 \begin{equation}\label{sequencesrelativesituation}
  \begin{split}
   &0\rightarrow I\rightarrow\bigoplus_{j=1}^mR(-e_j)\xrightarrow{u_1,\dots,u_m} M\rightarrow0, \ \text{ and }\\
   &0\rightarrow N\rightarrow\bigoplus_{i=1}^{n}R(-d_i)\xrightarrow{f_1,\dots,f_n}I\rightarrow0.
  \end{split}
 \end{equation}
\par Tensoring these sequences with the flat $\mathbb{Z}$-module $\mathbb{Q}$ we obtain exact sequences of $R_0$-modules. On the other hand if we apply the functor $-\otimes_{\mathbb{Z}}\mathbb{Z}/p\mathbb{Z}$ to the sequences \eqref{sequencesrelativesituation}, exactness is preserved for all primes except for a finite number of them. Let $h$ be the product of those primes, and we consider the smooth projective curve $Y=\mathrm{Proj}R_h$  over $\mathrm{Spec}\mathbb{Z}_h$. 
\par Let $U=D(R_{h+})$ denote the relative punctured spectrum. The sheaf $\widetilde{I}|_{U}$ restricts to $U_0=U\cap\mathrm{Spec}R_0$ as a locally free sheaf. By possibly shrinking the set $D(h)$ we may assume that $\widetilde{I}|_{U}$ is locally free. By further shrinking we may assume that $\mathcal{E}:=\widetilde{N}|_{U}$ is also locally free. Then for almost all  $p$, $I_p$ and $N_p$ are locally free on $U_p$.
\par Let $\mathcal{S}$, $\mathcal{T}$, $\mathcal{Q}$ be the locally free sheaves on $Y$ corresponding to $\mathcal{E}$, $\bigoplus_{i=1}^nR(-d_i)$,  $\widetilde{I}|_{U}$, which, by the second sequence of \eqref{sequencesrelativesituation}, form an exact sequence  $0\rightarrow\mathcal{S}\rightarrow\mathcal{T}\rightarrow\mathcal{Q}\rightarrow0$ on $Y$. Its restrictions give the short exact sequences $0\rightarrow\mathcal{S}_0\rightarrow\mathcal{T}_0\rightarrow\mathcal{Q}_0\rightarrow0$ on the generic fiber $Y_0$, and  $0\rightarrow\mathcal{S}_p\rightarrow\mathcal{T}_p\rightarrow\mathcal{Q}_p\rightarrow0$ on the fiber $Y_p$, for $p \nmid h$. 
\par Let $p$ be a prime number not dividing $h$, then we are in the situation of Theorem \ref{TheoremgHKgradedcase}, so we obtain
\begin{equation*}
  e_{gHK}^{R_p}(M_p)=\frac{1}{2\deg Y_p}\left(\mu_{HK}(\mathcal{S}_p)-\mu_{HK}(\mathcal{T}_p)+\mu_{HK}(\mathcal{Q}_p)\right).
\end{equation*}
 Then taking the limit for $p\rightarrow+\infty$ and applying Lemma \ref{trivedilemma} concludes the proof.
\end{proof}

\section*{Acknowledgements}
We would like to thank Mohsen Asgharzadeh for a careful reading of an earlier version of this article, and Hailong Dao for many helpful conversations. We thank the referee for showing us how to simplify the proofs of Lemma \ref{reflexivecohomologyzero} and Lemma \ref{HKclassgroup}.
Moreover, we thank the DFG \emph{Graduiertenkolleg Kombinatorische Strukturen in der Geometrie} at Osnabr\"uck for support.

\end{document}